\title{A quasi-commutativity property
of the Poisson and composition operators
\label{vers.7}
}
\author{A. Cialdea 
\thanks{Dipartimento di Matematica, Universit\`a della Basilicata, Viale 
dell'Ateneo Lucano 10, 85100, Potenza, Italy. \textit{email:}
cialdea@email.it.}  
\and V. Maz'ya
\thanks{
Department of Mathematical Sciences, M{\&}O Building,
University of Liverpool, Liverpool L69 7ZL, UK, and
Department of Mathematics,
  Link\"oping University, 
  SE-581 83 Link\"oping, Sweden.
\textit{email:} vlmaz@mai.liu.se.}
}
\date{}    

\documentclass[12pt]{article}

\usepackage{amssymb}
\usepackage{mathrsfs}
\usepackage{theorem}
\def\math{\mathscr}

\newtheorem{theorem}{Theorem}
\newtheorem{lemma}{Lemma}

\newenvironment{proof} {{\bf Proof.}}{\hfill \fbox{}\\ \smallskip}

\makeatletter
\renewcommand{\@makefnmark}{\hbox
{\@textsuperscript{\normalfont(\@thefnmark)}}}
\makeatother
      
{\theorembodyfont{\rmfamily} }

\newcommand{\reff}[1]{{\rm(\ref{#1})}}

\def\grande{\displaystyle}

\newcommand{\beginrighe}[1]{\begin{eqnarray}\label{#1}\begin{array}{c}\grande}
\def\endrighe{\end{array}\end{eqnarray}}

\def\R{\mathbb{R}}
\def\N{\mathbb{N}}

\def\de{\partial}

\def\a{\alpha}

\def\g{\gamma}
\def\G{\Gamma}
\def\d{\delta}

\def\e{\varepsilon}

\def\h{\eta}
\def\l{\lambda}

\def\n{\nu}
\def\ro{\varrho}
\def\s{\sigma}

\def\t{\tau}

\def\F{\Phi}
\def\O{\Omega}
\def\o{\omega}

\def\tr{\mathop{\rm tr}\nolimits}

\def\Dom{{\math D}}

\def\sign{\mathop{\rm sign}\nolimits}

\def\dive{\mathop{\rm div}\nolimits}

\def\leq{\leqslant}
\def\geq{\geqslant}

\def\tr{\hbox{\rm tr }}

\begin{document}

    \maketitle

\bigskip

    {\small  {\bf Abstract.} 
    Let $\F$ be a real valued function of one real variable,
   let $L$ denote
    an elliptic second order formally 
    self-adjoint differential operator with bounded measurable 
    coefficients, and let $P$ stand for
    the Poisson operator for $L$. A necessary and sufficient
    condition on $\F$ ensuring the equivalence of the Dirichlet 
    integrals of $\F\circ Ph$ and $P(\F\circ h)$ is obtained.
    We illustrate this result by some sharp inequalities
    for harmonic functions.}

    \bigskip \bigskip

\section{Introduction}

In the present article we consider an elliptic second order formally 
self-adjoint differential operator $L$ in a bounded domain $\O$.
We denote by $Ph$ the $L$-harmonic function with the Dirichlet data
$h$ on $\de\O$. The Dirichlet integral corresponding to the operator 
$L$ will be denoted by $\Dom[u]$. We also introduce a real-valued 
function $\F$ on the line $\R$ and denote the composition of 
$\F$ and $u$ by $\F\circ u$.

We want to show that the Dirichlet integrals of the functions
$\F\circ Ph$ and $P(\F\circ h)$ are comparable. First of all, 
clearly, the inequality
  $$
       \Dom[P(\F\circ h)[ \leq \Dom[\F\circ Ph]
 $$
   is valid. 
    Hence  we only need to check the opposite estimate
   \begin{equation}
       \Dom[\F\circ Ph] \leq C\,  \Dom[P(\F\circ h)]\, .
       \label{eq:opposite}
   \end{equation}
   
   We  find a condition on $\F$ which is both necessary and
   sufficient for \reff{eq:opposite}.
   
   Moreover, we prove that the two Dirichlet integrals are 
   comparable if and only if the derivative $\Psi=\F'$ 
   satisfies
   the reverse Cauchy inequality
   \begin{equation}
       {1\over b-a}\int_{a}^{b}\Psi^{2}(t)\, dt 
       \leq C \left({1\over b-a}\int_{a}^{b}\Psi(t)\, dt \right)^{2}
       \label{eq:rci}
   \end{equation}
   for any interval $(a,b)\subset \R$.
  
We add that the constants $C$ appearing in \reff{eq:opposite} and 
\reff{eq:rci} are the same.

At the end of the paper this result is illustrated
for harmonic functions and
for $\Psi(t)=|t|^{\a}$  with $\a>-1/2$. In particular,
we obtain
the sharp inequalities
\begin{equation}
    \int_{\O}|\nabla(|Ph|\,Ph)|^{2}dx\leq
    {3\over 2}\int_{\O}|\nabla P(|h|\,h)|^{2}dx
    \label{eq:intr1}
\end{equation}
for any  $h\in W^{1/2,2}(\de\O)$  and
\begin{equation}
    \int_{\O}|\nabla(Ph)^{2}|^{2}dx\leq
    {4\over 3} \int_{\O}|\nabla P(h^{2})|^{2}dx
    \label{eq:intr2}
\end{equation}
for any nonnegative  $h\in W^{1/2,2}(\de\O)$. Here $P$ is 
the harmonic Poisson operator.

To avoid technical complications connected with non-smoothness of the
boundary,  we only deal with domains 
bounded by surfaces of class $C^{\infty}$, although, in principle,
this restriction can be significantly weakened.

\section{Preliminaries}

All functions in this article are assumed to take real values
and the notation $\de_{i}$ stands for $\de/\de x_{i}$.

Let $L$ be the second order differential operator
$$
Lu=-\de_{i}(a_{ij}(x)\, \de_{j}u)
$$
defined in a bounded domain $\O\subset\R^{n}$.
 
The coefficients 
$a_{ij}$  are measurable and bounded. The operator $L$ is uniformly 
 elliptic, i.e. there exists $\l>0$ such that
 \begin{equation}
     a_{ij}(x)\xi_{i}\xi_{j}\geq \l\, |\xi|^{2}
     \label{eq:ellipt}
 \end{equation}
 for all $\xi\in\R^{n}$ and for almost every
 $x\in\O$.

Let $\Psi$ be a  function 
defined on $\R$ such that, for any $N\in \N$, the functions
\begin{equation}
    \Psi_{N}(t)=\cases{\Psi(t) & if  $\Psi(t)|\leq N$\cr
    N \sign (\Psi(t)) & if $|\Psi(t)|>N$}
    \label{eq:defpsik}
\end{equation}
are continuous.
We suppose that there exists a constant $C$
such that,
for any finite interval $\s\subset \R_{+}$, we
have
\begin{equation}
    \overline{\Psi^{2}} \leq  C (\overline{\Psi})^{2}
    \label{eq:Psi}
\end{equation}
where $\overline{u}$ denotes the mean value of $u$ on $\s$.

Also let
$$
\F(t)=\int_{0}^{t}\Psi(\t)\, d\t \, ,\quad
t\in\R.
$$

Let $W^{1/2,2}(\de\O)$ be the trace space for the Sobolev space
$W^{1,2}(\O)$ and
let $P$ denote the Poisson operator, i.e. the solution operator:
$$W^{1/2,2}(\de\O)\ni h \to u\in W^{1,2}(\O)$$
 for the Dirichlet problem
\begin{equation}
    \cases{Lu=0 & in $\O$\cr
    \tr u = h  & on $\de\O$}
    \label{eq:dir}
\end{equation}
where $\tr u$ is the trace of the function $u\in W^{1,2}(\O)$
on $\de\O$.

We introduce the Dirichlet integral
$$
\Dom[u]=\int_{\O}a_{ij}\de_{i}u\, \de_{j}u\, dx
$$
and the bilinear form
$$
\Dom[u,v]=\int_{\O}a_{ij}\de_{i}u\, \de_{j}v\, dx\, .
$$

In the sequel we shall consider 
$\Dom[P(\F\circ h)]$ and $\Dom[\F\circ Ph]$ for $h\in W^{1/2,2}(\de\O)$.
Since $h\in W^{1/2,2}(\de\O)$
 implies neither  $\F\circ h\in  W^{1/2,2}(\de\O)$
 nor $\F\circ Ph\in W^{1,2}(\O)$,
we have to specify what $\Dom[P(\F\circ h)]$ and $\Dom[\F\circ Ph]$
mean. 

We define
\begin{equation}
    \Dom[P(\F\circ h)]=\liminf_{k\to\infty}
    \Dom[P(\F_{k}\circ h)].
    \label{eq:defmin}
\end{equation}
where 
\begin{equation}
    \F_{k}(t)=\int_{0}^{t}\Psi_{k}(t)\, dt\, .
    \label{eq:defFn}
\end{equation}
and $\Psi_{k}$ is given by \reff{eq:defpsik}.
Note that if $h$ be in $W^{1/2,2}(\de\O)$, 
$\Dom[P(\F_{k}\circ h)]$ makes sense. In fact
$$
|\F_{k}\circ h|\leq k\, |h| \ \ \hbox{\rm and }\ 
|\F_{k}\circ h(x)-\F_{k}\circ h(y)| \leq k\, |h(x)-h(y)|,
$$
imply that $\F_{k}\circ h$ belongs
to $W^{1/2,2}(\de\O)$.

In order to accept definition \reff{eq:defmin},
we have to show that  if the left hand side 
of \reff{eq:defmin} makes sense because 
$\F \circ h$ belongs to $W^{1/2,2}(\de\O)$, then
\reff{eq:defmin} holds. In fact, we have:
\begin{lemma}
    Let $h\in W^{1/2,2}(\de\O)$ be such that also
$\F \circ h$ belongs to $W^{1/2,2}(\de\O)$. Then 
\begin{equation}
    \Dom[P(\F\circ h)]=\lim_{k\to\infty}\Dom[P(\F_{k}\circ h)].
    \label{eq:lemma}
\end{equation}
\end{lemma}

\begin{proof}
Obviously,
$$
\F_{k}\circ h(x)-\F\circ h(x)=\int_{0}^{h(x)}[\Psi_{k}(t)-\Psi(t)]\, dt
\to 0 \quad a.e.
$$
and
$$
|\,[\F_{k}\circ h(x)-\F\circ h(x)]-[\F_{k}\circ h(y)-\F\circ h(y)]\,| 
\leq 2\, |\F\circ h(x)-\F\circ h(y)|\, .
$$

In view of the Lebesgue dominated convergence theorem, these
inequalities
imply  $\F_{k}\circ h\to \F\circ h$ in $W^{1/2,2}(\de\O)$. Therefore,
$P(\F_{k}\circ h)\to P(\F\circ h)$ in $W^{1,2}(\O)$ and \reff{eq:lemma}
holds.
\end{proof}

As far as $\Dom[\F\circ Ph]$ is concerned, we remark that
$$
\Dom[\F_{k}\circ Ph]=\int_{\O}(\Psi_{k}(Ph))^{2}a_{ij}\de_{i}(Ph)\de_{j}(Ph)\, dx
$$
tends to
$$
\int_{\O}(\Psi(Ph))^{2}a_{ij}\de_{i}(Ph)\de_{j}(Ph)\, dx
$$
because of the monotone convergence theorem. Therefore,
we set
$$
    \Dom[\F\circ Ph]=\lim_{k\to\infty}\Dom[\F_{k}\circ Ph].
$$

Note that neither $\Dom[\F\circ Ph]$ nor $\Dom[P(\F\circ h)]$ 
 needs
to be finite.

Lety $G(x,y)$ be the Green function of the Dirichlet
problem \reff{eq:dir} and ${\de/\de\n}$ the
co-normal operator
	   $$
	   {\de \over \de \n}= a_{ij} \cos(n,x_{j})\,
	   \de_{i}
	   $$
	  where $n$ is the exterior unit normal.
	  
	  \begin{lemma}
	      Let  the coefficients $a_{ij}$ of the
	      operator $L$ belong to $C^{\infty}(\overline{\O})$. There exist two positive
	      constants $c_{1}$ and $c_{2}$ such that
	      \begin{equation}
	          {c_{1}\over |x-y|^{n}} \leq {\de^{2}G(x,y)\over \de 
		  \n_{x}\de \n_{y}} \leq {c_{2}\over |x-y|^{n}}
	          \label{eq:ineqG}
	      \end{equation}
	      for any $x,y \in \de\O$, $x\neq y$.
	  \end{lemma}
	  
	  \begin{proof}
	      Let us fix a point $x_{0}$ on $\de\O$.
	      We consider a neighborhood of $x_{0}$
	  and introduce local coordinates $y=(y',y_{n})$
	 in such a way that $x_{0}$ corresponds to $y=0$,
	 $y_{n}=0$ is the tangent hyperplane and locally $\O$
	 is contained in the half-space $y_{n}>0$.
	 We may suppose that this change of variables is such that
	      $a_{ij}(x_{0})=\d_{ij}$.

	      It is known (see \cite{MP})
	      that the  Poisson kernel
	      $(\de/ \de\n_{y})G(x,y)$ in a neighborhood of $x_{0}$
	      is given by
	      $$
	      2\, \o_{n}^{-1}\, y_{n} |y|^{-n}
	      + {\cal O}\left(|y|^{2-n-\e}\right)
	      $$
	      where $\o_{n}$ is the measure of the unit sphere
	      in $\R^{n}$ and $\e>0$.
	      Moreover the derivative of the Poisson kernel with
	      respect to $y_{n}$ is equal to
	      $$
	      2\, \o_{n}^{-1}\, (|y|^{2}-y_{n}^{2})|y|^{-n-2} +
	      {\cal O}\left(|y|^{1-n-\e}\right)
	      $$
	      which becomes
	      \begin{equation}
		  2\, \o_{n}^{-1}\,  |y'|^{-n} + 
		  {\cal O}\left(|y|^{1-n-\e}\right)
	          \label{eq:asympt}
	      \end{equation}
	      for $y_{n}=0$. 
	      Formula \reff{eq:asympt} and the arbitrariness
	      of $x_{0}$ imply  \reff{eq:ineqG}. 
	  \end{proof}

\section{The main result}

\begin{theorem}\label{th:main}
    If $\Psi:\R\to \R_{+}$ satisfies condition \reff{eq:Psi}, then
    \begin{equation}
	\Dom[\F\circ Ph]\leq C\, \Dom[P(\F\circ h)]
	\label{eq:goal}
    \end{equation}
    for any $h\in W^{1/2,2}(\de\O)$, where $C$ is the constant 
    in \reff{eq:Psi}. 
 
\end{theorem}

\begin{proof}   
We suppose temporarily that $a_{ij}\in C^{\infty}$.

 Let  $u$ be a  solution of the 
 equation $Lu=0$, $u\in C^{\infty}(\overline\O)$. 
 We show that
 \begin{equation}
     \Dom[u]={1\over 2}\int_{\de\O}\int_{\de\O}(\tr u(x)-\tr u(y))^{2}
      {\de^{2}G(x,y)\over \de\n_{x}\de\n_{y}} d\s_{x}d\s_{y}.
     \label{eq:relQ}
 \end{equation}
 
 In fact, since
 $$
 L_{x}[(u(x)-u(y))^{2}]=2\, a_{hk}\de_{h}u\, \de_{k}u\, ,
 $$
   the integration by parts in \reff{eq:relQ} gives
   $$\displaylines{
   \int_{\de\O}\int_{\de\O}(\tr u(x)-\tr u(y))^{2}
	 {\de^{2}G(x,y)\over \de\n_{x}\de\n_{y}} d\s_{x}d\s_{y}=\cr
	 2\int_{\O}a_{hk}\de_{h}u\, \de_{k}u\, dx
	 \int_{\de\O}{\de G(x,y)\over \de \n_{y}}\, d\s_{y}=
	 2\Dom[u]\, ,
	 }
   $$
   and \reff{eq:relQ} is proved.

   Let $u\in W^{1,2}(\O)$ be a solution of 
   $Lu=0$ in $\O$ and let $\{u_{k}\}$ be a sequence of
   $C^{\infty}(\overline{\O}$ functions which tends to 
   $u$ in $W^{1,2}(\O)$.
   Since $\tr u_{k}\to\tr u$ in $W^{1/2,2}(\de\O)$, we see that
   $P(\tr u_{k})$ tends to $P(\tr u)=u$ in $W^{1,2}(\O)$
   and therefore $\Dom[P(\tr u_{k})]\to \Dom[u]$. This implies that
   \reff{eq:relQ} holds for any $u$  in $W^{1,2}(\O)$ 
   with
   $Lu=0$ in $\O$.

   Let now $u$ and $v$ belong to $W^{1,2}(\O)$ and $Lu=Lv=0$ in $\O$.
   Since
   $$
   \Dom[u,v]=4^{-1}(\Dom[u+v]-\Dom[u-v]),
   $$
   we can write
   \begin{equation}
       \Dom[u,v]=
         {1\over 2} \int_{\de\O}\int_{\de\O}(\tr u(x)-\tr u(y))
	 (\tr v(x)-\tr v(y))
               {\de^{2}G(x,y)\over \de\n_{x}\de\n_{y}} d\s_{x}d\s_{y}.
       \label{eq:B1}
   \end{equation}
   
   Note also that, if $h\in W^{1/2,2}(\de\O)$ and $g\in W^{1,2}(\O)$,
   we have
   \begin{equation}
       \Dom[Ph, P(\tr g)]=\Dom[Ph,g].
       \label{eq:B2}
   \end{equation}
   
    Suppose now that $h\in W^{1/2,2}(\de\O)$ 
    is such that $\F\circ h\in W^{1/2,2}(\de\O)$. We have
    $$\displaylines{
    \Dom[\F\circ Ph]=\cr
    \int_{\O}a_{ij}\de_{i}(\F\circ Ph)\de_{j}(\F\circ Ph)\, dx
    = \int_{\O}a_{ij}(\Psi(Ph))^{2}\de_{i}(Ph)\de_{j}(Ph)\, dx\, .
    }
    $$
    
    The last integral can be written as
    $$
    \int_{\O}a_{ij}\de_{i}(Ph)\de_{j}\left(\int_{0}^{Ph}
    \Psi^{2}(\t)\, d\t\right)dx\, ,
    $$
    and we have proved that 
    $$\Dom[\F\circ Ph]= \Dom\left(Ph, \int_{0}^{Ph}
    \Psi^{2}(\t)\, d\t\right)
    $$

    From \reff{eq:B1} and \reff{eq:B2}, we get
    \begin{equation}
       \Dom[\F\circ Ph]= \displaystyle {1\over 2} \int_{\de\O}\int_{\de\O}(h(y)-h(x))\int_{h(x)}^{h(y)}
	   \Psi^{2}(\t)\, d\t {\de^{2}G(x,y)\over \de\n_{x}\de\n_{y}}
	   d\s_{x}d\s_{y}\,  
        \label{eq:B=}
    \end{equation}
    
    In view of 
	\reff{eq:ineqG}, $\de^{2}G(x,y)/\de\n_{x}\de\n_{y}$ is positive, 
   and the condition \reff{eq:Psi} leads to
  \begin{equation}
        \Dom[\F\circ Ph] \leq {C\over 2} 
          \int_{\de\O}\int_{\de\O}\left(\int_{h(x)}^{h(y)}
                \Psi(\t)\, d\t\right)^{2}
                {\de^{2}G(x,y)\over \de\n_{x}\de\n_{y}}
                d\s_{x}d\s_{y}\, .
      \label{eq:N}
  \end{equation}
    
    Inequality \reff{eq:goal} is proved, since the right-hand side 
    in \reff{eq:N} is 
    nothing but $C\, \Dom[P(\F\circ h)]$ (see \reff{eq:relQ}).
    
    For any $h\in W^{1/2,2}(\de\O)$, the inequality
    \reff{eq:goal} follows from
    $$
    \Dom[\F\circ Ph]= \lim_{n\to\infty}\Dom[\F_{k}\circ Ph] \leq
    C\liminf_{n\to\infty} \Dom[P(\F_{k}\circ h)].
    $$

Let us suppose now that $a_{ij}$ only belong to $L^{\infty}(\O)$. 
There exist $a_{ij}^{(k)}\in C^{\infty}(\R^{n})$ such that
$a_{ij}^{(k)}\to a_{ij}$ 
in measure as $k\to \infty$.
We can assume that
$$
\Vert a_{ij}^{(k)}\Vert_{L^{\infty}(\O)}\leq K
$$
and that the operators $L^{(k)}=-\de_{i}(a_{ij}^{(k)}\de_{j}u)$ 
satisfy the ellipticity condition \reff{eq:ellipt}
with the same constant $\l$.

Let $u$ be a solution of the Dirichlet problem
\reff{eq:dir} and  let $u_{k}$ satisfy
$$
    \cases{L^{(k)}u_{k}=0 & in $\O$\cr
       \tr u_{k} = h  & on $\de\O$\, . 
       }
$$

Denote the matrices  $\{a_{ij}\}$ and $\{a_{ij}^{(k)}\}$ 
by $A$ and $A^{(k)}$ respectively.

Since we can write
$$
\dive A\nabla (u-u_{k})=  -\dive A \nabla u_{k} = 
 - \dive(A-A_{k})\nabla u_{k}
$$
we find that
$$
\Dom[u-u_{k}]\leq \Vert (A-A_{k}) \nabla u_{k}\Vert \, \Vert 
\nabla(u-u_{k})\Vert\, .
$$

Then there exists a constant $K$ such that
\begin{equation}
    \Vert \nabla(u-u_{k})\Vert \leq 
    K\, \Vert (A-A_{k})\nabla u_{k}\Vert\, .
    \label{eq:nabla}
\end{equation}

Denoting by $\Dom_{k}$ the quadratic form
$$
\Dom_{k}[u]=\int_{\O}a_{ij}^{(k)}\de_{i}u\, \de_{j}u\, dx\, ,
$$
we have
$$
\Dom_{k}[u_{k}]= 
\min_{u\in W^{1,2}(\O)\atop 
{\scriptscriptstyle\rm tr}\, u = h} \Dom_{k}[u]
$$
and
$$
\l\, \Vert \nabla u_{k}\Vert^{2} \leq \Dom_{k}[u_{k}].
$$

This shows that the sequence $\Vert \nabla u_{k}\Vert$
is bounded and the right-hand side
of \reff{eq:nabla} tends to $0$ as 
$k\to \infty$.

We are now in a position to prove \reff{eq:goal}. 
Clearly, it is enough 
to show that
\begin{equation}
    \Dom[\F_{m}\circ Ph]\leq C\, \Dom[P(\F_{m}\circ h)]\, ,
    \label{eq:goaln}
\end{equation}
where $\F_{m}$ is given by \reff{eq:defFn} for any $h\in 
W^{1/2,2}(\de\O)$ such that $\F\circ h$ belongs to the same space.

Because of what we have proved when the coefficients are
smooth, 
we may write
$$
\Dom_{k}[\F_{m}\circ P_{k}h]\leq C\, \Dom_{k}[P_{k}(\F_{m}\circ h)]\, ,
$$
where $P_{k}$ denotes the Poisson operator for $L^{(k)}$.

Formula \reff{eq:nabla}
shows that $P_{k}(\F_{m}\circ h)$ tends to $P(\F_{m}\circ h)$
(as $k\to 
\infty$) in $W^{1,2}(\O)$ and thus
\begin{equation}
    \lim_{k\to\infty} \Dom_{k}[P_{k}(\F_{m}\circ h)] = \Dom[P(\F_{m}\circ h)]\, .
    \label{eq:ref1}
\end{equation}

On the other hand, we have
$$\displaylines{
\nabla \F_{m}(P_{k}h)- \nabla \F_{m}(Ph)=\cr
\Psi_{m}(P_{k}h)(\nabla P_{k}h-\nabla P h) +
(\Psi_{m}(P_{k}h)-\Psi_{m}(Ph))\, \nabla Ph\, .
}
$$

In view of the continuity of $\Psi_{m}$, we find
$$
\Vert \nabla \F_{m}(P_{k}h)- \nabla \F_{m}(Ph)\Vert_{L^{2}(\O)}\to 0.
$$

This implies that $\Dom_{k}[\F_{m}\circ P_{k}h] \to \Dom[\F_{m}\circ Ph]$, which
together with \reff{eq:ref1}, leads to 
\reff{eq:goaln}.
\end{proof}

Under the assumption that the coefficients of the operator are smooth,
we can prove the inverse of  Thorem \ref{th:main}.

\begin{theorem}\label{th:main2}
    Let  the coefficients of the operator $L$ belong to 
    $C^{\infty}(\overline{\O})$. If \reff{eq:goal} holds
for any $h\in W^{1/2,2}(\de\O)$, then \reff{eq:Psi} is true
with the same constant $C$.
\end{theorem}

\begin{proof}
    Let $\G$ be a subdomain of $\de\O$ with smooth non-empty 
    boundary. We choose a sufficiently small $\e>0$ and denote the
    $\e$-neighborhood of $\G$ by $[\G]_{\e}$. 
   We set $\g_{\e}=[\G]_{\e}\setminus\G$ and
    denote by $\d(x)$ the distance of the point $x$ from $\G$.

    Let $a$ and $b$ be different real numbers
      and let $h$ be the function defined on
     $\de\O$ by
    $$
        h(x)\cases{=a & if $x\in\G$\cr
	=a+(b-a)\, \e^{-1}\d(x) & if $x\in \g_{\e}$\cr
	=b & if $x\in\de\O\setminus [\G]_{\e}$.}
    $$

    We know from \reff{eq:B=} that $\Dom[\F\circ Ph]$ is equal to
    $$
    \int_{\de\O}d\s_{x}\int_{\de\O}Q(x,y)\, d\s_{y}
    $$
    where
    $$
    Q(x,y)={1\over 2}\, (h(y)-h(x))\left(\int_{h(x)}^{h(y)}
	   \Psi^{2}(\t)\, d\t\right) {\de^{2}G(x,y)\over 
	   \de\n_{x}\de\n_{y}}\, .
    $$

We can write
 $$
 \int_{\de\O}d\s_{x}\int_{\de\O}Q(x,y)\, d\s_{y}= I_{1}+I_{2}+I_{3}
 $$
where
$$
\displaylines{
I_{1}=\int_{\g_{\e}}d\s_{x}\int_{\g_{\e}}Q(x,y)\, d\s_{y}\, ,\cr
I_{2}=\int_{\de\O\setminus \g_{\e}}d\s_{x}\int_{\g_{\e}}Q(x,y)\, d\s_{y}
+ \int_{\g_{\e}}d\s_{x}\int_{
\de\O\setminus \g_{\e}}Q(x,y)\, d\s_{y}\, ,
\cr
I_{3}=\int_{\G}d\s_{x}\int_{\de\O\setminus [\G]_{\e}}Q(x,y)\, d\s_{y}
+ \int_{\de\O\setminus [\G]_{\e}}d\s_{x}\int_{\G}Q(x,y)\, d\s_{y}\, .
}
$$

The right-hand estimate in \reff{eq:ineqG} leads to 
\begin{equation}
    I_{1}\leq c\left({b-a\over 
    \e}\right)^{2}\int_{\g_{\e}}d\s_{x}
    \int_{\g_{\e}}(\d(x)-\d(y))^{2}|x-y|^{-n}d\s_{y}.
    \label{eq:0}
\end{equation}

The integral
 $$
 \int_{\g_{\e}}(\d(x)-\d(y))^{2}|x-y|^{-n}d\s_{y}
 $$
with $x\in\g_{\e}$ is majorized by
 $$
c \int_{0}^{\e}(t-\d(x))^{2}dt
\int_{\R^{n-2}}(|\h|+(t-\d(x)))^{-n}d\h={\cal O}(\e)\, .
 $$

 This estimate and \reff{eq:0}
 imply $I_{1}={\cal O}(1)$ as $\e\to 0^{+}$.

 Since
  $$
  \left|\int_{\G}d\s_{x}\int_{\g_{\e}}Q(x,y)\, d\s_{y}\right| \leq
  c\left({b-a\over 
      \e}\right)^{2}\int_{\G}d\s_{x}\int_{\g_{\e}}\d^{2}(y)|x-y|^{-n}d\s_{y}
  $$
  and the integral
  $$
  \int_{\g_{\e}}\d^{2}(y)d\s_{y}\int_{\G}|x-y|^{-n}d\s_{x}
  $$
  does not exceed
  $$
  c\int_{0}^{\e}t^{2}dt\int_{\R^{n-2}}(|\h|+t)^{-n+1}d\h = {\cal 
  O}(\e^{2})\, ,
  $$
  we find
  \begin{equation}
      \int_{\G}d\s_{x}\int_{\g_{\e}}Q(x,y)\, d\s_{y}= {\cal 
       O}(1).
      \label{eq:*}
  \end{equation}

  Analogously, 
  \begin{eqnarray}
  & \displaystyle \left|\int_{\de\O\setminus [\G]_{\e}}d\s_{x}\int_{\g_{\e}}Q(x,y)\,
  d\s_{y}\right| 
  \leq & \nonumber \\
  &\displaystyle  c\left({b-a\over 
      \e}\right)^{2}\int_{\de\O\setminus [\G]_{\e}}d\s_{x}\int_{\g_{\e}}(\e-\d(y))^{2}|x-y|^{-n}d\s_{y}
   ={\cal O}(1). &
   \label{eq:**}
  \end{eqnarray}
  
  Exchanging the roles of $x$ and $y$ in the previous argument,
  we arrive at the estimate
   $$
   \int_{\g_{\e}}d\s_{x}\int_{
   \de\O\setminus \g_{\e}}Q(x,y)\, d\s_{y}
 =  {\cal O}(1)\, .
   $$
   
   Combining this with \reff{eq:*} and \reff{eq:**},
    we see that $I_{2}={\cal O}(1)$.

    Let us consider $I_{3}$. Since the two terms in the definition
     of $I_{3}$ are equal, we have
     $$
     I_{3}=2\, (b-a)\int_{a}^{b}\Psi^{2}(\t)\, d\t
     \int_{\G}d\s_{x}\int_{\de\O\setminus [\G]_{\e}}
     {\de^{2}G(x,y)\over 
     \de\n_{x}\de\n_{y}}\, d\s_{y}\, .
     $$
     
     By the left inequality in \reff{eq:ineqG},
    $$
	 I_{3}\geq 2\, c_{1}(b-a)\int_{a}^{b}\Psi^{2}(\t)\, d\t
	  \int_{\G}d\s_{x}\int_{\de\O\setminus [\G]_{\e}} |x-y|^{-n}d\s_{y}\, .
$$

    There exists $\ro>0$ such that the integral over
    $\G\times (\de\O\setminus [\G]_{\e})$ of $|x-y|^{-n}$ admits
    the lower estimate by
   $$
   c\int_{-\ro}^{0}dt\int_{\e}^{\ro}ds \int_{|\t|\leq\ro}d\t 
   \int_{|\h|\leq \ro}
   (|\h-\t|+ |s-t|)^{-n}d\h
   $$
which implies
\begin{equation}
    \int_{\G}d\s_{x}\int_{\de\O\setminus [\G]_{\e}}
	 {\de^{2}G(x,y)\over 
	 \de\n_{x}\de\n_{y}}\, d\s_{y}\geq c\, \log(1/\e).
     \label{eq:I3dis}
 \end{equation}
 
 Now we deal with $\Dom[P(\F\circ h)]$. 
 This can be written as $J_{1}+J_{2}+J_{3}$, where 
 $J_{s}$ are defined as $I_{s}$ ($s=1,2,3$) with the 
 only difference that $Q(x,y)$ is replaced by
 $$
{1\over 2} \left(\int_{h(x)}^{h(y)}
	    \Psi(\t)\, d\t\right)^{2} {\de^{2}G(x,y)\over 
	    \de\n_{x}\de\n_{y}}\, .
 $$
 
 As before, 
 $J_{1}={\cal O}(1)$, $J_{2}={\cal O}(1)$ and
 $$
     J_{3}=2\, \left(\int_{a}^{b}\Psi(\t) d\t\right)^{2}
	  \int_{\G}d\s_{x}\int_{\de\O\setminus [\G]_{\e}}
	  {\de^{2}G(x,y)\over 
	  \de\n_{x}\de\n_{y}}\, d\s_{y}\, .
$$

Inequality \reff{eq:goal} for the function $h$
can be written as
$$
I_{3}+{\cal O}(1) \leq C\, (J_{3}+{\cal O}(1))
$$
i.e.
$$
\displaylines{
{\cal O}(1)+ (b-a)\int_{a}^{b}\Psi^{2}(\t)\, d\t
     \int_{\G}d\s_{x}\int_{\de\O\setminus [\G]_{\e}}
     {\de^{2}G(x,y)\over 
     \de\n_{x}\de\n_{y}}\, d\s_{y}  \leq \cr
    C\left( {\cal O}(1)+ \left(\int_{a}^{b}\Psi(\t) d\t\right)^{2}
	       \int_{\G}d\s_{x}\int_{\de\O\setminus [\G]_{\e}}
	       {\de^{2}G(x,y)\over 
	       \de\n_{x}\de\n_{y}}\, d\s_{y}  
	       \right).
	       }
$$

Dividing both sides by 
$$
\int_{\G}d\s_{x}\int_{\de\O\setminus [\G]_{\e}}
	       {\de^{2}G(x,y)\over 
	       \de\n_{x}\de\n_{y}}\, d\s_{y} 
$$
and letting $\e\to 0^{+}$ we arrive at \reff{eq:rci},
referring to \reff{eq:I3dis}.
The proof is complete.
\end{proof}

\textbf{Remark.}  Inspection of the proofs of Theorems 
\ref{th:main} and \ref{th:main2}
shows that if the coefficients are smooth and 
 Dirichlet data on $\de\O$ are nonnegative, 
we have also the following result:

\noindent\textit{The inequality
$$
\Dom[\F\circ Ph]\leq C_{+}\, \Dom[P(\F\circ h)]
$$
 holds   for any nonnegative $h\in W^{1/2,2}(\de\O)$
    if and only if
$$
\overline{\Psi^{2}} \leq  C_{+} (\overline{\Psi})^{2}
$$
for all finite intervals $\s\subset \R_{+}$. }

\textbf{Example.} As a simple application of this Theorem,
consider the case of the  Laplace operator
and the function  $\Psi(t)=|t|^{\a}$ ($\a> -1/2$). 
Let $C_{\a}$ and $C_{\a,+}$ be the following constants:
\begin{eqnarray}
& \displaystyle C_{\a} ={(\a+1)^{2}\over 2\a+1}\, 
\sup_{t\in\R}{(1-t)(1-t^{2\a +1})\over (1- |t|^{\a}t)^{2}}
& \label{eq:Ca}\\
& \displaystyle C_{\a,+} ={(\a+1)^{2}\over 2\a+1}\, 
\sup_{t\in\R_{+}}{(1-t)(1-t^{2\a +1})\over (1- t^{\a+1})^{2}}\, .
& \label{eq:Ca+}
\end{eqnarray}

Theorem \ref{th:main} shows that
\begin{equation}
    \int_{\O}|\nabla(|Ph|^{\a}Ph)|^{2}dx\leq
    C_{\a}\int_{\O}|\nabla P(|h|^{\a}h)|^{2}dx\, ,
    \label{eq:ex}
\end{equation}
for any $h\in W^{1/2,2}(\de\O)$, 
where $P$ denotes the  harmonic Poisson operator for $\O$.

In view of the Remark above, we have also
\begin{equation}
    \int_{\O}|\nabla(Ph)^{\a+1}|^{2}dx\leq
    C_{\a,+}\int_{\O}|\nabla P(h^{\a+1})|^{2}dx\, ,
    \label{eq:ex+}
\end{equation}
for any nonnegative $h\in W^{1/2,2}(\de\O)$. 
Because of Theorem \ref{th:main2}, the constants $C_{\a}$ 
and $C_{\a,+}$ are the best possible in \reff{eq:ex} and 
\reff{eq:ex+}. 

Hence the problem of finding explicitly the best constant in such
inequalities is   reduced to the determination of the supremum
in \reff{eq:Ca} and \reff{eq:Ca+}.

One can check that $C_{\a,+}=(\a+1)^{2}/(2\a+1)$.
In the particular case $\a=1$, 
$C_{1}$ can also be determined. This leads to the inequalities
\reff{eq:intr1} and \reff{eq:intr2}.

\end{document}